\documentclass[12pt]{article}
\usepackage{latexsym}
\usepackage{amsmath,amsthm}
\usepackage{amssymb}
\usepackage{amsfonts}
\usepackage{url}

\def\claim#1{\begin{trivlist}\item[\hskip\labelsep\bf#1]\it}
\def\endclaim{\end{trivlist}}

\numberwithin{equation}{section}

\newtheorem{theorem}{Theorem}[section]
\newtheorem{lemma}[theorem]{Lemma}

\newtheorem{definition}{Definition}[section]
\newtheorem{remark}{Remark}[section]

\newcommand{\C}{{\mathbb C}}

\newcommand{\ep}{\epsilon}

\newcommand{\im}{{\mbox{\rm Im\ }}}
\newcommand{\R}{{\mathbb R}}

\title{Strong unique continuation for two-dimensional anisotropic elliptic systems}
\author{Rulin Kuan\thanks{Department of Mathematics, National Cheng Kung University, Tainan, Taiwan. 
\newline (Email: rkuan@mail.ncku.edu.tw) Partially supported by the Ministry of Science and Technology, Taiwan
under project MOST 105 - 2115 - M - 006 - 017 - MY2.};\hspace{1cm}
Gen Nakamura\thanks{Department of Mathematics, Hokkaido University, Hokkaido, Japan. \newline (Email: nakamuragenn@gmail.com)};\hspace{1cm}
Satoshi Sasayama\thanks{Department of Mathematics, Hokkaido University, Hokkaido, Japan.
\newline (Email: sasayama@math.sci.hokudai.ac.jp)}}
\date{}
\begin{document}
\renewcommand{\theequation}{\thesection.\arabic{equation}}

\maketitle
\begin{abstract}
In this paper, we give the strong unique continuation property for a general two dimensional anisotropic elliptic system with real coefficients in a Gevrey class under the assumption that the principal symbol of the system has simple characteristics. The strong unique continuation property is derived by obtaining some Carleman estimate. The derivation of the Carleman estimate is based on transforming the system to a larger second order elliptic system with diagonal principal part which has complex coefficients.    
\end{abstract}

\section{Introduction}\label{sec1}
\setcounter{equation}{0}

In this paper we are concerned with the \emph{strong unique continuation
property} ( abbreviated by SUCP) for two dimensional anisotropic elliptic systems with real coefficients in a Gevrey class under the assumption that the principal part of the systems have simple characteristics. 

Let $\Omega\subset \mathbb{R}^2$ be a domain. We denote by $u(x)=[u_1(x), u_2(x)]^{\top}$ a two-component 
vector-valued function, where $^{\top}$ denotes the transpose. Let  $u(x)$ be a solution of 
\begin{equation}\label{eq0}
Lu+Fu=0\quad\text{in}\quad\Omega,
\end{equation}
where
$$(Lu)_{\alpha}(x)=\sum_{\beta=1}^2\sum_{j,\ell=1}^2A_{\alpha\beta}^{j\ell}(x)\partial_j\partial_{\ell}u_{\beta}\,\,\text{with $\partial_j=\partial/\partial x_j$}
$$
and
$$
(Fu)_{\alpha}(x)=\sum_{\beta=1}^2\sum_{\ell=1}^2B_{\alpha\beta}^{\ell}(x)\partial_{\ell}u_{\beta}+\sum_{\beta=1}^2C_{\alpha\beta}(x)u_{\beta}
$$
for $\alpha=1,2$. We assume that  $A_{\alpha\beta}^{j\ell}$ satisfies the major symmetry given as
\begin{equation}\label{symm}
A_{\alpha\beta}^{j\ell}(x)=A_{\beta\alpha}^{\ell
j}(x), \,\, x\in\Omega,\; 1\le j,\ell\le 2,\, 1\le\alpha,\beta\le 2
\end{equation}
and the strong ellipticity given as follows. There exists $\delta>0$ such that for any real vectors
$a=[a_1,a_2]^{\top}$ and $b=[b_1,b_2]^{\top},$ we have
\begin{equation}\label{elliptic}
\sum_{\alpha,\beta=1}^N\sum_{j,\ell=1}^2A_{\alpha\beta}^{j\ell}(x){a}_{\alpha}b_ja_{\beta}b_{\ell}\geq\delta|a|^2|b|^2,\,\,
x\in\Omega.
\end{equation}

A well-known example of \eqref{eq0} is the elasticity system where 
in addition to \eqref{symm}, $A_{\alpha\beta}^{j\ell}$ also satisfies the
minor symmetry given as
\begin{equation}\label{symm2}
A^{j\ell}_{\alpha\beta}(x)=A^{j\beta}_{\alpha\ell}(x),\,\,
x\in\Omega,\,\, 1\le\alpha,\beta,j,\ell\le 2.
\end{equation}

\begin{definition}${}$
\newline
Let $P$ be a partial differential operator defined in $\Omega$
\newline
{\rm (i)}
We say that $P$ has the \emph{unique continuation property} (abbreivated by UCP) in $\Omega$ if $u$ satisfies $Pu=0$ in $\Omega$ and it vanishes in a subdomain of $\Omega$, then $u=0$ in $\Omega$.
\newline
{\rm (ii)} We say that $P$ has the \emph{strong unique continuation property} in $\Omega$ if $u$ satisfies $Pu=0$ in $\Omega$ and $u$ vanishes at a point  $x_0\in\Omega$ of infinite order, i.e. $u$ satisfies for all $N>0$
\[
\int_{B_{r}(x_0)}|u|^2 dx=O(r^N), \quad\mbox{ as }r\to 0,
\] 
then $u=0$ in $\Omega$. 
\end{definition}


For the isotropic elasticity systems, the UCP has been proved in \cite{almo}, \cite{dero}, \cite{aity}, 
\cite{we1} and \cite{we2}. Unlike the isotropic case, very little
is known for the anisotropic elasticity systems. In this direction, the UCP and SUCP for
an elasticity system with residual stress was shown in
\cite{naw1} and \cite{lin} respectively. It is known that this system is no longer isotropic
due to the existence of residual stress. For other related results on the UCP, we refer the reader to \cite{pasinuw}.

For the system \eqref{eq0}, a UCP result was proved in \cite{naw2d}, where in addition to the smoothness assumption, $A_{\alpha\beta}^{j\ell}(x)\in
C^{0,1}(\Omega), B_{\alpha\beta}^{\ell}(x),
C_{\alpha\beta}(x)\in L^{\infty}(\Omega)$, some restrictions on the
characteristic roots and associated eigenvectors are required. In this paper we want to establish SUCP for \eqref{eq0} and simplify the restrictions on the characteristic roots. 
An essential ingredient in our proofs is to reduce the system
\eqref{eq0} to a second order system with diagonal principal part, which will be done by enlarging the size of the vector $u$ to
four. Our method of reduction of  \eqref{eq0} has evolved from a series
of papers \cite{lw}, \cite{naw2d}, and \cite{pasinuw} where the
system \eqref{eq0} was reduced to a first order elliptic system of
two variables. In this paper we will make further reduction and
transform so that \eqref{eq0} can be reduced to a second order system with diagonal principal part (see \eqref{eqnew}). To be more precise, the leading part of the
new system is
$$
\begin{bmatrix}\mathcal{P}_1(x,D)I&0\\0&\mathcal{P}_2(x,D)I\end{bmatrix},
$$
where
$\mathcal{P}_1(x,D)$ and $\mathcal{P}_2(x,D)$ are second order elliptic operators with \emph{complex} coefficients satisfying $$\mathcal{P}_2(x,D)=\overline{\mathcal{P}_1}(x,D),$$ where $\overline{\mathcal{P}_1}$ is the complex conjugate of $\mathcal{P}_1$.  Moreover, we want to remark that for any $x_0\in\Omega$ and $\xi\neq 0$ neither $\mathcal{P}_1(x_0,\xi)$ nor $\mathcal{P}_2(x_0,\xi)$ is real-valued. For such operators with lower order terms, SUCP is in general not true (see  \cite{al}. \cite{ab}). Nonetheless, when coefficients of $\mathcal{P}_1$ or $\mathcal{P}_2$ are Gevrey class of certain order, the strong uniqueness continuation property can be restored \cite{cgt}, \cite{le}. One of the main steps of our proof for Theorem~\ref{main} is to derive suitable Carleman estimates. Our derivation of the Carleman estimate is based on that given in \cite{le} for elliptic, second order scalar differential operators with Gevrey coefficients.

\medskip
Without loss of generality, $x_0$ is assumed to be the origin throughout the paper. In order to state our main result, we prepare the definitions of Gevrey class and the simple (complex) characteristics for system of differential operators.

\begin{definition}[Gevrey class]\label{Gevrey}
Given $\Omega\subset\mathbb{R}^n$, $n\in \mathbb{N}$ and $\sigma>1$, the Gevrey class $G^\sigma=G^{\sigma}(\Omega)$ of order $\sigma$ consists all functions $f\in C^{\infty}(\Omega)$ such that for every compact subset $K\subset\Omega$, there are constants $C$ and $R$ with 
\[
|\partial^{\alpha}g(x)|\leq C R^{|\alpha|}(|\alpha|!)^{\sigma},
\]
for all multi-indices $\alpha$ and $x\in K$.
\end{definition}
We remark that when $\sigma=1$, $G^1$ is the set of analytic functions $C^{\omega}(\Omega)$. That means $C^{\omega}(\Omega)\subset G^{\sigma}\subset C^{\infty}(\Omega)$ for $\sigma>1$.

\begin{definition}\label{simple}
Let $L$ be an elliptic operator and $\ell$ be the principal symbol of $L$. Then we say $L$ has simple characteristics in $\Omega$, if $\det(\ell)$ satisfies
\[
|\det(\ell)(x,\zeta)|+\sum_{j=1}^2\Big|\frac{\partial\det(\ell)(x,\zeta)}{\partial\zeta_j}\Big|\neq 0
\]
at each $x\in\Omega$ for any non-zero arbitrary complex vector $\zeta=(\zeta_1,\zeta_2)^{\top}$.
\end{definition}
We remark that this definition is similar to that given in \cite{sit} for scalar partial differential equations. 

\medskip
Now we state our main result.
\begin{theorem}\label{main}
Assume \eqref{eq0} has $G^{\sigma}$ coefficients with $\sigma>1$ and satisfies \eqref{symm} and \eqref{elliptic}. 
If $L$ has simple characteristics in $\Omega$, then there exists $\nu_0$ such that for any $\sigma<1+\frac{1}{\nu_0}$, SUCP holds for any solution $u\in G^\sigma$ of \eqref{eq0}.
\end{theorem}

We want to point out some connections between Theorem~\ref{main} and Calder\'on's type uniqueness theorems. 
Due to the ellipticity condition \eqref{elliptic}, four characterisitic roots corresponding to the symbol of $L$ appear in two complex conjugate pairs. Hence the multiplicity of the characteric roots is at most two. This is the assumption in most Calder\'on's type theorems \cite{zuily}. However, we cannot quarantee the smoothness of characteristic roots without putting extra restrictions on $A_{\alpha\beta}^{j\ell}(x)$. In fact, even with infinitely smooth $A_{\alpha\beta}^{j\ell}(x)$, characteristic roots, in general, depend on $x$ {\em discontinuously}. In view of the counterexample given by Pli\'s \cite{plis}, any attempt to prove UCP without putting any assumption on the characteristic roots is doomed to fail. Hence we put some condition on the characteristic roots as stated in Theorem \ref{main}. 

The rest of this paper is organized as follows. In Section 2 we reduce (1.1) to a second order system with diagonal principal part by enlarging the size of $u$. The Carleman estimate is given in Section 3. By using this estimate, SUCP for (1.1) is proved in Section 4. Finally we give some general two dimensional anisotropic elasticity system which satisfies the assumptions of Theorem \ref{main}.

\section{Reduction to a system with a diagonal leading term}\label{sec2}
\setcounter{equation}{0}

In this section we transform \eqref{eq0} into a second order elliptic system with diagonal principal part. We first follow the reduction given in \cite{naw2d} to transform \eqref{eq0} to a larger first order system. Rewrite \eqref{eq0} into
\begin{equation}\label{luu}
Lu+Fu=\Lambda_{11}\partial_1^2u+\Lambda_{12}\partial_1\partial_2u+\Lambda_{22}\partial_2^2u+Fu=0,
\end{equation}
where
$$\Lambda_{11}=[A_{\alpha\beta}^{11}],\  \Lambda_{12}=[A_{\alpha\beta}^{21}]+[A_{\alpha\beta}^{21}]^T ,\ \Lambda_{22}=[A_{\alpha\beta}^{22}].$$
Suppose that $T$ satisfies 
\begin{equation}\label{T}
T^2-\Lambda_{11}^{-1}\Lambda_{12}T+\Lambda_{11}^{-1}\Lambda_{22}=0
\end{equation}
and $\Psi$ satisfies 
\begin{equation}\label{syl}
\Psi T-(\Lambda_{11}^{-1}\Lambda_{12}-T)\Psi+M=0,
\end{equation}
where 
\[
M=-\Lambda_{11}^{-1}B_1 T+\Lambda_{11}^{-1}B_2-\Lambda_{11}^{-1}\Lambda_{12}\partial_2 T+T\partial_2 T-\partial_1 T
\]
and $B_1=[B^1_{\alpha\beta}]$, $B_2=[B^2_{\alpha\beta}]$. Then in terms of 
\begin{align}\label{W}
W=[w_1,w_2]^{\top}=[u,\Psi u+\partial_1u+T\partial_2u]^{\top},
\end{align}
\eqref{eq0} becomes
\begin{equation}\label{block}
\partial_1 W+{\cal M}_1\partial_2 W={\cal M}_0 W
\end{equation}
with
$$
{\cal M}_1=\begin{bmatrix}T&0\\
0&\Lambda_{11}^{-1}\Lambda_{12}-T\end{bmatrix}.
$$
We remark here that UCP for \eqref{eq0} was proved by reducing it to \eqref{block} in \cite{naw2d} under the condition that $A_{\alpha\beta}^{j\ell}\in C^{0,1}(\Omega)$ and $B_{\alpha\beta}^{\ell},C_{\alpha\beta}\in L^{\infty}(\Omega)$. To understand what is $T$, let $P(\lambda)$ be a matrix polynomial defined by
\begin{equation}\label{P_def}
P(\lambda):=\Lambda_{11}\lambda^2+\Lambda_{12}\lambda+\Lambda_{22}.
\end{equation}
Clearly $P(\lambda)$ is positive definite for any $\lambda\in\mathbb{R}$. Then it is well known that  $P(\lambda)$ can be factorized as
\begin{equation}\label{factor}
P(\lambda)=(\lambda-X^{\ast})\Lambda_{11}(\lambda-X)
\end{equation}
with
\begin{equation}\label{X}
X=(\oint_{\Gamma_+}\zeta
P(\zeta)^{-1}d\zeta)(\oint_{\Gamma_+}P(\zeta)^{-1}d\zeta)^{-1},
\end{equation}
where $\Gamma_+\subset\C_+=\{\mbox{Im}\,z>0\}$ is a closed contour enclosing all
spectra of $P$ in $\C_+$. Further, the above factorization is unique under the condition $\text{Spec}(X)\subset\C_+$, where $\text{Spec}(X)$ denotes the spectrum of $X$. See Threom 2.4 in \cite{ito} for reference.

Hence, we can simply choose $T=-X$ and thus $T$ is also invertible. With such choice of $T$, we now take a look
at $\Lambda_{11}^{-1}\Lambda_{12}-T$. Straightforward
computations give
\[
\begin{array}{rl}
\Lambda_{11}^{-1}P(\lambda)&=\lambda^2+\Lambda_{11}^{-1}\Lambda_{12}\lambda+\Lambda_{11}^{-1}\Lambda_{22}\\
{}&=\lambda^2+\Lambda_{11}^{-1}\Lambda_{12}\lambda+\Lambda_{11}^{-1}\Lambda_{12}T-T^2\\
{}&=(\lambda+\Lambda_{11}^{-1}\Lambda_{12}-T)(\lambda+T).
\end{array}
\]
Combining the above results with \eqref{factor}, we have
\begin{equation}\label{eq1011}
\Lambda_{11}^{-1}(\lambda+T^{\ast})\Lambda_{11}(\lambda+T)
=(\lambda+\Lambda_{11}^{-1}\Lambda_{12}-T)(\lambda+T)
\end{equation}
It is easy to see from \eqref{eq1011} that
\begin{equation}\label{cong}
\overline{\text{Spec}(T)}=\text{Spec}(\Lambda_{11}^{-1}\Lambda_{12}-T)
\end{equation}
and therefore
\begin{equation}\label{noeig}
\text{Spec}(T)\cap\text{Spec}(\Lambda_{11}^{-1}\Lambda_{12}-T)=\emptyset.
\end{equation}

Next, to further simplify \eqref{block}, we denote
$$
T=\begin{bmatrix}t_{11}&t_{12}\\
t_{21}&t_{22}\end{bmatrix},\quad\text{cof}^\prime T=\begin{bmatrix}t_{22}&-t_{12}\\
-t_{21}&t_{11}\end{bmatrix},
$$
Here note that $\text{cof}^\prime T$ is the transpose of the cofactor matrix of $T$.
Similarly, we can define $\text{cof}^\prime(\Lambda_{11}^{-1}\Lambda_{12}-T)$. We now consider the composition of differential operators
$$
{\cal P}_1(x,D):=-(\partial_1+\text{cof}^\prime\
T\partial_2)(\partial_1+T\partial_2)
$$
and
$$
{\cal P}_2(x,D):=-(\partial_1+\text{cof}^\prime(\Lambda_{11}^{-1}\Lambda_{12}-T)
\partial_2)(\partial_1+(\Lambda_{11}^{-1}\Lambda_{12}-T)\partial_2),
$$
where $D=(D_1, D_2)$ with each $D_j$ given by $D_j=-i \partial_j$. It is clear that the principal symbols of ${\cal P}_1$ and ${\cal P}_2$ are
$$
(\xi_1^2+\text{tr}T\xi_1\xi_2+\text{det}T\xi_2^2)I
$$
and
$$
(\xi_1^2+\text{tr}(\Lambda_{11}^{-1}\Lambda_{12}-T)\xi_1\xi_2+\text{det}(\Lambda_{11}^{-1}\Lambda_{12}-T)\xi_2^2)I
$$
respectively. Also, it is easy to see that $\mathcal{P}_1$ and $\mathcal{P}_2$ have complex coefficients.

Having these arguments in mind, we apply the differential operators
$$
-\partial_1-\begin{bmatrix}\text{cof}^\prime\, T&0\\ 0&\text{cof}^\prime\
(\Lambda_{11}^{-1}\Lambda_{12}-T)\end{bmatrix}\partial_2
$$
on both sides of \eqref{block} and obtain that
\begin{equation}\label{eqnew}
{\cal P}(x,D)W:=\begin{bmatrix}{\cal P}_1(x,D)I&0\\0&{\cal P}_2(x,D)I\end{bmatrix}W=\sum_{|\alpha|\leq
1}K_{\alpha}(x)D^{\alpha}W,
\end{equation}
where
\begin{equation}\label{P1}
{\cal P}_1(x,D)=D_1^2+\text{tr}TD_1D_2+\text{det}TD_2^2,
\end{equation}
and
\begin{equation}\label{P2}
{\cal P}_2(x,D)=D_1^2+\text{tr}(\Lambda_{11}^{-1}\Lambda_{12}-T)D_1D_2+\text{det}(\Lambda_{11}^{-1}\Lambda_{12}-T)D_2^2.
\end{equation}
In view of \eqref{cong}, we have that
$$\text{tr}T=\overline{\text{tr}(\Lambda_{11}^{-1}\Lambda_{12}-T)}\quad \text{and}\quad \text{det}T=\overline{\text{det}(\Lambda_{11}^{-1}\Lambda_{12}-T)},$$
which implies that ${\cal P}_2(x,D)=\overline{{\cal P}_1}(x,D)$. 

In the next section, we assume that the coefficients of \eqref{eq0} are in some Gevrey class and prove SUCP for \eqref{eq0}. By the derivation of $\cal{P}$, SUCP for \eqref{eq0} will follow if we prove that for \eqref{eqnew}. The proof will rely on suitable Carleman estimate for $\cal{P}$, which is the main theme of the next section.

\section{Carleman estimate for $\mathcal{P}$}\label{sec3}
\setcounter{equation}{0}
In this section, we assume that the coefficients $A_{\alpha\beta}^{j\ell}, B_{\alpha\beta}^{\ell}, C_{\alpha\beta}$ belong to some Gevrey class and we will derive a suitable Carleman estimate for ${\cal P}$ defined by \eqref{eqnew}. Although ${\cal P}$ has a diagonal principal part whose diagonal elements consists of scalar, simple characterisitics, second order differential operators with complex coefficients, these kind of partial differential operators in the diagonal elements do not have SUCP in general. Neverthelees,  if the coefficients are in some Gevrey class, it is given in \cite{le} that the differential operator in each diagonal element can have SUCP by deriving the following Carleman estimate. 
\begin{lemma}\label{L_result}
Let $P$ be a second order elliptic operator with $C^{\infty}$ coefficients which are defined in an neighborhood $\Omega$ of $0$ in $\R^2$. Let $p$ be the principal symbol of $P$. Suppose $p$ has simple characteristics in $\Omega$. Then there exists a real number $\nu_0>0$, depending only on $p(0,\R^2\setminus\{0\})$, such that for $\nu>\nu_0$ and $v\in C^{\infty}(\Omega)$, we have the following Carleman estimate. That is there exist $C,R_0>0$ and large enough $\tau_0>0$ such that for $\tau>\tau_0$, 
\begin{equation}\label{v_est}
\tau\|r^{-\nu-2}v\|_{L^2(B_{R_0}(0))}
+\|r^{-1}\nabla v\|_{L^2(B_{R_0}(0))}
\leq C \|P_{\tau}v\|_{L^2(B_{R_0}(0))},
\end{equation}
where $P_{\tau}=e^{\frac{\tau}{\nu}r^{-\nu}}P e^{-\frac{\tau}{\nu}r^{-\nu}}$, $r=|x|$ and $C,R_0,\tau_0$ are independent of $v$ and $\tau$.
\end{lemma}

\begin{remark}${}$
\newline
{\rm (i)}  Observe that $p(0,\mathbb{R}^2\setminus\{0\})$ is either a convex cone of $\mathbb{C}\setminus\{0\}$ or it is $\mathbb{C}\setminus\{0\}$. If it is a convex cone  with angle $2\varphi$, $0\leq\varphi<\frac{\pi}{2}$,  $\nu_0$ can be taken more precisely as
\[
\nu_0\geq \frac{2\sin\varphi}{1-\sin\varphi}.
\]
In either case, there exists $\nu_0$ as stated in Lemma \ref{L_result}. See \cite{cgt}, \cite{le} for these. In our case, the cone we have for each principal symbols of $\mathcal{P}_j$, $j=1,2$, is ${\mathbb C}\setminus\{0\}$.
\newline
{\rm (ii)} It is worth to mention that the idea of proving this lemma is to write $P=X_1X_2+Q$ with first order differential operators $X_1,X_2,Q$ with $C^{\infty}$ coefficients. This smoothness of the coefficients follows from the assumption that $P$ is simple characteristics, and this smoothness is very important for proving this lemma.  

\end{remark}

By using the above lemma, we can establish a Carleman estimate suitable to derive SUCP for our system \eqref{eqnew}. To do so, we need some lemmas.

First of them is that the above estimate \eqref{v_est} fits nicely to Gevrey functions. Actually, Gevrey functions have the following property (see  \cite{cgt} and \cite{le}).
\begin{lemma}\label{flat}
Let $\Omega$ be an open neighborhood of $0$ in $\R^n$ and $u\in G^{\sigma}$. 
If $u$ is flat at the origin (i.e. $\partial_x^{\alpha}u=0$ at the origin for all multi-index $\alpha$), then there exists $v\in C^\infty(\Omega)$ flat at the origin such that
\[
u=\exp(-|x|^{-\nu})v
\]
provided $1+\nu^{-1}>\sigma$, i.e. $\nu<\frac{1}{\sigma-1}$. 
\end{lemma}
In order to restate Lemma~\ref{L_result} using Lemma \ref{flat}, we need to show $\mathcal{P}_1,\mathcal{P}_2$ are second order elliptic with $C^{\infty}(\Omega)$ coefficients and have simple characteristics. It is easy to check that $\mathcal{P}_1$ and $\mathcal{P}_2$ are elliptic operators due to \eqref{symm} and \eqref{elliptic}. The rest of properties of $\mathcal{P}_1$ and $\mathcal{P}_2$ are shown in the following lemma.
\begin{lemma}\label{G}
Suppose the coefficients of \eqref{eq0} are $G^{\sigma}$ functions, i.e.  $A_{\alpha\beta}^{j\ell},B_{\alpha\beta}^{\ell},C_{\alpha\beta}\in G^{\sigma}$, for some $\sigma>1$ and for all $1\leq\alpha,\beta,j,\ell\leq 2$. Then $T=-X$ and $\Psi$, where $X$ and $\Psi$ satisfy \eqref{X} and \eqref{syl} respectively, are both in $G^{\sigma}$. As a consequence $P_1$ and $P_2$ which are defined respectively in \eqref{P1} and \eqref{P2}, also have $G^{\sigma}$ coefficients.
\end{lemma}
\begin{proof}
We first we prove $T\in G^{\sigma}$. Recall that $T=-X$ and $X$ satisfies \eqref{X}, i.e. for any fixed $x\in\Omega$,
\[
X(x)=(\oint_{\Gamma_+(x)}\zeta
P(x,\zeta)^{-1}d\zeta)(\oint_{\Gamma_+(x)}P(x,\zeta)^{-1}d\zeta)^{-1},
\]
where $P(x,\zeta)$ is $P(\lambda)$ with $\lambda=\zeta$ given by \eqref{P_def} clarifying its dependency on $x\in\Omega$, 
$\Gamma_+(x)$ is a contour in $\C_+=\{\im \zeta>0\}$ enclosing all the roots of $\text{det} P(x,\zeta)=0$ in $\zeta\in\C_+$ counter-clockwise which depends on $x\in\Omega$. Note that  $\Gamma_+(x)$ can be taken locally invariant, because the integrands are holomorphic with respect to $\zeta$. Together with this and the well known fact that $G^\sigma$ is a field with respect to addition and multiplication which are consistent to each other, we can easily see that
\[\int_{\Gamma_+(x)}P(x,\zeta)^{-1}d\zeta,\,\,\int_{\Gamma_+(x)}\zeta P(x,\zeta)^{-1}d\zeta\in G^\sigma.\] Hence it suffices to show
\[
\Big(\int_{\Gamma_+(x)}P(x,\zeta)^{-1}d\zeta\Big)^{-1}\in G^\sigma.
\]  
For this we only need to know the invertibility of $\int_{\Gamma_+(x)}P(x,\zeta)^{-1}d\zeta$ which is guaranteed by Theorem 2.4 of \cite{ito}. We also can see Section 3.7 in \cite{wloka_book} for another reference. 
\par
Next we prove $\Psi\in G^{\sigma}$. Recall that \eqref{syl}, i.e. $\Psi$ satisfies
\[
\Psi T-(\Lambda_{11}^{-1}\Lambda_{12}-T)\Psi+M=0,
\]
Concerning the unique solvability of this equation, we consider a more general equation 
\begin{equation}\label{13}
\Psi A-B\Psi+C=0.
\end{equation}
Write 
\[
\Psi=\begin{pmatrix}
a&b\\
c&d
\end{pmatrix}\quad\mbox{and}\quad
C=\begin{pmatrix}
c_1&c_2\\
c_3&c_4
\end{pmatrix},
\]
then \eqref{13} can be written as 
\begin{equation}\label{13.1}
M_{AB}\begin{pmatrix}
a\\b\\c\\d
\end{pmatrix}
+\begin{pmatrix}
c_1\\c_2\\c_3\\c_4
\end{pmatrix}=0,
\end{equation}
where $M_{AB}$ is a $4\times 4$ matrix, whose entries depend only on the summation, subtraction and multiplication of some entries of $A$ and $B$.
We observe that
\[
\begin{aligned}
&\text{Spec}(A)\cap \text{Spec}(B)=\emptyset\\
\Leftrightarrow &\forall \mbox{ matrix }C,\, \exists!\,\Psi \mbox{ solves }\eqref{13}\\
\Leftrightarrow &\forall \mbox{ matrix }C,\, \exists!\,\Psi \mbox{ solves }\eqref{13.1}\\
\Leftrightarrow & \det(M_{AB})\neq 0
\end{aligned}.
\]
Therefore if $\text{Spec}(A)\cap \text{Spec}{B}=\emptyset$, we have
\[
\begin{pmatrix}
a\\b\\c\\d
\end{pmatrix}=-M_{AB}^{-1}\begin{pmatrix}
c_1\\c_2\\c_3\\c_4
\end{pmatrix}=-\frac{\text{cof}^\prime\,(M_{AB})}{\det(M_{AB})}\begin{pmatrix}
c_1\\c_2\\c_3\\c_4
\end{pmatrix}.
\]
Hence, if $A,B,C$ are in $G^{\sigma}$, $\Psi$ is also in $G^{\sigma}$.

Let $A=T$, $B=\Lambda_{11}^{-1}\Lambda_{12}-T$ and $C=M$. Then due to \eqref{noeig} and our assumption, the coefficients of \eqref{eq0} are in $G^{\sigma}$, we obtain $A,B,C\in G^{\sigma}$, and hence $\Psi\in G^{\sigma}$. 

\end{proof}

\begin{lemma}\label{P_simple}
Suppose $L$ satisfies \eqref{symm} and \eqref{elliptic} . If $L$ has simple characteristics in $\Omega$, then $\mathcal{P}_1$ and $\mathcal{P}_2$ both have simple characteristics in $\Omega$, where $\mathcal{P}_1,\mathcal{P}_2$ are defined in \eqref{P1} and \eqref{P2}.
\end{lemma}
\begin{proof}
From \eqref{simple}, $L$ has simple characteristics in $\Omega$ means 
\[
\det\ell(x,\zeta)\mbox{ has distinct roots in }\zeta\,\,\text{for each $x\in \Omega$}.
\]
where $\ell$ is the principal symbol of $L$. We notice that by \eqref{P_def}, in fact $P(x,\lambda)=\ell(x;\zeta)$ with $\zeta=(\lambda,1)$. Therefore $\det P(x,\lambda)$ has distinct roots  in $\lambda$ for each $x\in\Omega$,
Moreover, from the factorization of $P(x,\lambda)$, see \eqref{factor}, i.e.
\[
P(x,\lambda)=(\lambda-X^{\ast})\Lambda_{11}(\lambda-X)\,\,\text{with $X$ which depends on $x$},
\]
it is easy to see that $\text{Spec}(X)$ is included in the set of all the roots of ${\rm det} P(\lambda,x)=0$ with respect to $\lambda\in \mathbb{C}$. Therefore if $\text{det}P(x, \lambda)=0$ in $\lambda$ has distinct roots, $X$ has distinct eigenvalues and so does $T=-X$ for each fixed $x\in\Omega$. Denote the distinct eigenvalues of $T$ by $\lambda_1,\lambda_2$. Then by \eqref{P1} and \eqref{P2}, we can rewrite 
\[\begin{aligned}
{\cal P}_1(x,D)&=D_1^2+(\lambda_1+\lambda_2)D_1D_2+\lambda_1\lambda_2\\
&=(D_1+\lambda_1D_2)(D_1+\lambda_2D_2)
\end{aligned}\]  
and 
\[
{\cal P}_2(x,D)=(D_1+\overline{\lambda}_1D_2)(D_1+\overline{\lambda}_2D_2)
\]
Since $\lambda_1\neq\lambda_2$, again by the definition \eqref{simple}, ${\cal P}_1,{\cal P}_2$ have simple characteristics. 
\end{proof}

Based on the above observations and lemmas, we have the following Carleman estimates for ${\cal P}$.

\begin{theorem}\label{Thm-Carleman}
Suppose \eqref{eq0} has $G^{\sigma}$ coefficients and $L$ satisfies \eqref{symm} and \eqref{elliptic}. Let $\cal P$ be the matrix differential operator induced by $L$ and defined in \eqref{eqnew} and $p_1,p_2$ be the principal symbols of $\mathcal{P}_1,\mathcal{P}_2$ respectively. Suppose $L$ has simple characteristics in $\Omega$. Then there exists $\nu_0>0$ depending only on $p_1(0,\R^2\setminus\{0\})$ and $p_2(0,\R^2\setminus\{0\})$ and $\tau_0>0$ such that if $\nu_0<\frac{1}{\sigma-1}$, then for $\nu_0<\nu<\frac{1}{\sigma-1}$ and $\tau>\tau_0$, we have 
\begin{equation}\label{main_Car}
\|e^{\frac{\tau}{\nu}|x|^{-\nu}}|x|^{-1}\nabla w\|_{L^2(B_{R_0}(0))}
+\tau\|e^{\frac{\tau}{\nu}|x|^{-\nu}}|x|^{-\nu-2}w\|_{L^2(B_{R_0}(0))}
\leq C \|e^{\frac{\tau}{\nu}|x|^{-\nu}}{\cal P}w\|_{L^2(B_{R_0}(0))}
\end{equation}
for $w\in G^{\sigma}$ flat at the origin, where $C$ and $R_0$ are independent of $w$ and $\tau$. 
\end{theorem}
\begin{proof}
From Lemma~\ref{G} and Lemma~\ref{P_simple}, we know $\mathcal{P}_1,\mathcal{P}_2$ both have $G^{\sigma}$ coefficients and simple characteristics. 
Therefore by applying Lemma~\ref{L_result} to $\mathcal{P}_1$ and $\mathcal{P}_2$, there exist $\tau_0,\nu_0>0$ such that for any $\nu>\nu_0$ and $v\in C^{\infty}(\Omega)$, we have
\begin{equation}\label{vj_est}
\tau\|r^{-\nu-2}v\|_{L^2(B_{R_0}(0))}
+\|r^{-1}\nabla v\|_{L^2(B_{R_0}(0))}
\leq C \|\mathcal{P}_{j,\tau}v\|_{L^2(B_{R_0}(0))}
\end{equation}
for some $C,R_0>0$, where $r=|x|$, $\mathcal{P}_{j,\tau}=e^{\frac{\tau}{\nu}|x|^{-\nu}}\mathcal{P}_je^{-\frac{\tau}{\nu}|x|^{-\nu}}$ and $C,R_0$ are independent of $v$ and $\tau$.

Now let $w\in G^{\sigma}$ be flat at the origin. By Lemma~\ref{flat}, for given $\nu<\frac{1}{\sigma-1}$, there exists $v_{\nu}\in C^{\infty}(\Omega)$ such that
\[
w=e^{-|x|^{-(\nu+\ep)}}v_{\nu},
\]
provided that $\ep>0$ satisfies $\nu+\ep<\frac{1}{\sigma-1}$.  
Hence we have
\[
w=e^{-\frac{\tau}{\nu}|x|^{-\nu}}v_{\tau,\nu},\,\,\tau>0
\]
with 
\[
v_{\tau,\nu}=e^{\frac{\tau}{\nu}|x|^{-\nu}-|x|^{-(\nu+\ep)}}v_{\nu}\in C^{\infty}.
\]
By substituting  $v=v_{\tau,\nu}=e^{\frac{\tau}{\nu}|x|^{-\nu}}w$ into \eqref{vj_est}, we have for $\tau>\tau_0$,
\[\begin{aligned}
\tau\||x|^{-\nu-2}\big(e^{\frac{\tau}{\nu}|x|^{-\nu}}w\big)\|_{L^2(B_{R_0}(0))}
&+\||x|^{-1}\nabla \big(e^{\frac{\tau}{\nu}|x|^{-\nu}}w\big)\|_{L^2(B_{R_0}(0))}\\
&\qquad\leq C \|\mathcal{P}_{j,\tau}\big(e^{\frac{\tau}{\nu}|x|^{-\nu}}w\big)\|_{L^2(B_{R_0}(0))},
\end{aligned}\]
and therefore
\[\begin{aligned}
\tau\|e^{\frac{\tau}{\nu}|x|^{-\nu}}|x|^{-\nu-2}w\|_{L^2(B_{R_0}(0))}
&+\|e^{\frac{\tau}{\nu}|x|^{-\nu}}|x|^{-1}\nabla w\|_{L^2(B_{R_0}(0))}\\
&\qquad\leq \widetilde{C} \|e^{\frac{\tau}{\nu}|x|^{-\nu}}\mathcal{P}_jw\|_{L^2(B_{R_0}(0))},
\end{aligned}\]
where $C,\widetilde{C},R_0$ are independent of $w$ and $\tau$. By recalling \eqref{eqnew}, we immediately have \eqref{main_Car}.
\end{proof}

\section{Proof of Theorem~\ref{main}}

We first show that $W\in G^\sigma$ defined by \eqref{W} is flat at the origin.
\begin{lemma}\label{W flat}
Suppose \eqref{eq0} has $G^{\sigma}$ coefficients, i.e. $A_{\alpha\beta}^{j\ell}, B_{\alpha\beta}^{\ell}, C_{\alpha\beta}$ are in $G^{\sigma}$. If $u$ is a solution of \eqref{eq0} and satisfies
\begin{equation}\label{*}
\int_{B_r(0)}|u|^2dx=O(r^N)\,\,\mbox{ for all }N \mbox{ as }r\to 0,
\end{equation}
then $W\in G^{\sigma}$ defined by \eqref{W}is flat at the origin.
\end{lemma}
\begin{proof}
By Theorem 17.1.4 in \cite{hormander}, if $u$ is a solution of \eqref{eq0} and satisfies \eqref{*}, then for $|\beta|\leq 2$, we have  
\[
\int_{r<|x|<2r}|r^{|\beta|}D^{\beta}u|^2dx=O(r^N)\,\,\mbox{ for all }N\mbox{ as }r\to 0.
\]
Hence we have  $D^{\beta}u(0)=0$ for $|\beta|\leq 2$. Then using \eqref{eq0},  we have $D^{\beta}u(0)=0$ for all $\beta$.
By recalling \eqref{W}, we immediately have $W\in G^\sigma$ is flat at the origin.

\end{proof}

We are now ready to prove Theorem~\ref{main}.
\begin{proof}[Proof of Theorem~\ref{main}]
By Theorem \ref{Thm-Carleman} and Lemma \ref{W flat}, there exists $R_0$ such that
\[
\begin{aligned}
&\|e^{\frac{\tau}{\nu}|x|^{-\nu}}|x|^{-1}DW\|_{L^2(B_{R_0}(0))}
+\tau\|e^{\frac{\tau}{\nu}|x|^{-\nu}}|x|^{-\nu-2}W\|_{L^2(B_{R_0}(0))}\\
&\leq C \|e^{\frac{\tau}{\nu}|x|^{-\nu}}{\cal P}W\|_{L^2(B_{R_0}(0))}\\
&\leq C\|e^{\frac{\tau}{\nu}|x|^{-\nu}}\sum_{|\alpha|\leq 1}K_{\alpha}D^{\alpha}W\|_{L^2(B_{R_0}(0))}
\end{aligned}
\]
for $\tau>\tau_0$ with some $\tau_0$, where $C$ is independent of $\tau$ and $W$. 
Since $K_{\alpha}\in G^\sigma,\,\,|\alpha|\le 1$, we have
\[
\begin{aligned}
&\|e^{\frac{\tau}{\nu}|x|^{-\nu}}|x|^{-1}DW\|_{L^2(B_{R_0}(0))}
+\tau\|e^{\frac{\tau}{\nu}|x|^{-\nu}}|x|^{-\nu-2}W\|_{L^2(B_{R_0}(0))}\\
&\leq C'\Big(\|e^{\frac{\tau}{\nu}|x|^{-\nu}}DW\|_{L^2(B_{R_0}(0))}
+\|e^{\frac{\tau}{\nu}|x|^{-\nu}}W\|_{L^2(B_{R_0}(0))}\Big).
\end{aligned}
\]
Choose $R_1<R_0$ such that $R_1^{-1}>2C'$. Then we have
\[
\begin{aligned}
&\|e^{\frac{\tau}{\nu}|x|^{-\nu}}|x|^{-1}DW\|_{L^2(B_{R_1}(0))}
+\tau\|e^{\frac{\tau}{\nu}|x|^{-\nu}}|x|^{-\nu-2}W\|_{L^2(B_{R_1}(0))}\\
&\leq C'\Big(\|e^{\frac{\tau}{\nu}|x|^{-\nu}}DW\|_{L^2(B_{R_1}(0))}
+\|e^{\frac{\tau}{\nu}|x|^{-\nu}}W\|_{L^2(B_{R_1}(0))}\\
&\qquad+\|e^{\frac{\tau}{\nu}|x|^{-\nu}}DW\|_{L^2\big(B_{R_0}(0)\setminus B_{R_1}(0)\big)}
+\|e^{\frac{\tau}{\nu}|x|^{-\nu}}W\|_{L^2\big(B_{R_0}(0)\setminus B_{R_1}(0)\big)}\Big).
\end{aligned}
\]
Thus 
\[
\|e^{\frac{\tau}{\nu}|x|^{-\nu}}DW\|_{L^2(B_{R_1}(0))} \quad\mbox{and}\quad
\|e^{\frac{\tau}{\nu}|x|^{-\nu}}W\|_{L^2(B_{R_1}(0))}
\]
can be absorbed by 
$\|e^{\frac{\tau}{\nu}|x|^{-\nu}}|x|^{-1}DW\|_{L^2(B_{R_1}(0))}$ and 
$\tau\|e^{\frac{\tau}{\nu}|x|^{-\nu}}|x|^{-\nu-2}W\|_{L^2(B_{R_1}(0))}$ respectively. That is,
\[
\begin{aligned}
&\frac{1}{2}\|e^{\frac{\tau}{\nu}|x|^{-\nu}}|x|^{-1}DW\|_{L^2(B_{R_1}(0))}
+\frac{1}{2}\tau\|e^{\frac{\tau}{\nu}|x|^{-\nu}}|x|^{-\nu-2}W\|_{L^2(B_{R_1}(0))}\\
&\leq C'\Big(
\|e^{\frac{\tau}{\nu}|x|^{-\nu}}DW\|_{L^2\big(B_{R_0}(0)\setminus B_{R_1}(0)\big)}
+\|e^{\frac{\tau}{\nu}|x|^{-\nu}}W\|_{L^2\big(B_{R_0}(0)\setminus B_{R_1}(0)\big)}\Big).
\end{aligned}
\]
Since $|x|\leq R_1$ in $B_{R_1}(0)$ and $|x|\geq R_1$ in $B_{R_0(0)}\setminus B_{R_1}(0)$, we have
\[
\begin{aligned}
&e^{\frac{\tau}{\nu}R_1^{-\nu}}R_1^{-1}\|DW\|_{L^2(B_{R_1}(0))}
+\tau e^{\frac{\tau}{\nu}R_1^{-\nu}}R_1^{-\nu-2}\|W\|_{L^2(B_{R_1}(0))}\\
&\leq C''\Big(
e^{\frac{\tau}{\nu}R_1^{-\nu}}\|DW\|_{L^2\big(B_{R_0}(0)\setminus B_{R_1}(0)\big)}
+e^{\frac{\tau}{\nu}R_1^{-\nu}}\|W\|_{L^2\big(B_{R_0}(0)\setminus B_{R_1}(0)\big)}\Big).
\end{aligned}
\] 
After cancelling out $e^{\frac{\tau}{\nu}R_1^{-\nu}}$, the right hand side of the above inequality becomes bounded as $\tau\rightarrow\infty$. Therefore by letting $\tau\to\infty$, we obtain $\|W\|_{L^2\big(B_{R_1}(0)\big)}=0$, i.e. $W=0$ in $B_{R_1}(0)$. From  \eqref{W}, it is readily seen that $u=0$ in $B_{R_1}(0)$. 
By an standard argument, the set $\{x\in\Omega: u(x)=0\mbox{ and } u \mbox{ is flat at } x\}$ is open and closed in $\Omega$. Since $\Omega$ is connected, we have $u=0$ in $\Omega$.
\end{proof}

\section{General example for Theorem \ref{main}}\label{sec5}
In this section we give some general elastic operator which satisfies the assumptions of Theorem \ref{main}. Let  $L$ be a $2\times 2$ partial differential operator defined by 
\begin{equation}\label{example L}
(Lu)_{\alpha}=\sum_{\beta=1}^2\sum_{j,\ell=1}^2
A^{j\ell}_{\alpha\beta}\partial_j\partial_{\ell}u_{\beta},
\end{equation}
where
\begin{equation}\label{15}
[A_{\alpha\beta}^{11}]=\begin{pmatrix}
a&b\\
b&c
\end{pmatrix}, \quad
[A_{\alpha\beta}^{12}]+[A_{\alpha\beta}^{21}]=\begin{pmatrix}
2b&c+d\\
c+d&2e
\end{pmatrix},\quad
[A_{\alpha\beta}^{22}]=\begin{pmatrix}
c&e\\
e&f
\end{pmatrix}
\end{equation}
with real valued functions $a, b ,c ,d, e, f\in G^\sigma$. Then $A_{\alpha\beta}^{j\ell}$ satisfies \eqref{symm} and \eqref{symm2}. 
  
\begin{lemma}
Assume that $a, c, ac-b^2>0$, 
\[
f>\max (\frac{b^2c}{a^2},\frac{2ab^2c-b^4}{a^3}).
\] 
and let 
\[
e=\frac{bc}{a},\qquad d=\frac{2b^2-ac}{a}.
\]
Then $L$ satisfies \eqref{elliptic} and it has simple characteristic.
\end{lemma}
\begin{proof}
We first prove that $L$ satisfies \eqref{elliptic}. For any vectors $\eta=[\eta_1,\eta_2]^{\top}\in\R^2$, $\xi=[\xi_1,\xi_2]^{\top}\in \R^2$, 
\[\begin{aligned}
&\sum_{\alpha,\beta=1}^2\sum_{j,\ell=1}^2 
A_{\alpha\beta}^{j\ell}\eta_{\alpha}\eta_{\beta}\xi_j\xi_{\ell}\\
&=\xi_1^2\eta^{\top}\begin{pmatrix}
a&b\\
b&c
\end{pmatrix}\eta
+\xi_1\xi_2\eta^{\top}\begin{pmatrix}
2b&c+d\\
c+d&2e
\end{pmatrix}\eta
+\xi_2^2\eta^{\top}\begin{pmatrix}
c&e\\
e&f
\end{pmatrix}\eta\\
&=\xi_1^2\eta^{\top}\begin{pmatrix}
a&b\\
b&c
\end{pmatrix}\eta
+\xi_1\xi_2\eta^{\top}\begin{pmatrix}
2b&c+\frac{2b^2-ac}{a}\\
c+\frac{2b^2-ac}{a}&\frac{2bc}{a}
\end{pmatrix}\eta
+\xi_2^2\eta^{\top}\begin{pmatrix}
c&\frac{bc}{a}\\
\frac{bc}{a}&f
\end{pmatrix}\eta\\
&=\eta^{\top}\begin{pmatrix}
a&b\\
b&c
\end{pmatrix}\eta\,
\xi^{\top}\begin{pmatrix}
1&\frac{b}{a}\\
\frac{b}{a}&
\frac
 {\frac{c}{a}\eta^{\top}\begin{pmatrix}
 a&b\\
 b&\frac{af}{c}
 \end{pmatrix}\eta
} 
{\eta^{\top}\begin{pmatrix}
 a&b\\
 b&c
 \end{pmatrix}\eta
}
\end{pmatrix}\xi.
\end{aligned}\]
Since $\begin{pmatrix}a&b\\b&c\end{pmatrix}$ is positive, it suffices to show 
\[
\begin{pmatrix}
1&\frac{b}{a}\\
\frac{b}{a}&
\frac
 {\frac{c}{a}\eta^{\top}\begin{pmatrix}
 a&b\\
 b&\frac{af}{c}
 \end{pmatrix}\eta
} 
{\eta^{\top}\begin{pmatrix}
 a&b\\
 b&c
 \end{pmatrix}\eta
}
\end{pmatrix}\mbox{ is positive.}
\]
Then using $ac-b^2>0$ and $f>\frac{2ab^2c-b^4}{a^3}$ which we have from the assumptions, it is not difficult to show that the above matrix is positive for any $\eta\neq (0,0)^{\top}$.

Next we prove that $L$ has simple characteristics. The principal symbol $\ell(x,\zeta)$ of $L$ and its determinant are given as
\[
\begin{aligned}
\ell(x,\zeta)&=\sum_{j,\ell}A^{j\ell}_{\alpha\beta}\zeta_j\zeta_{\ell}\\
&=\begin{pmatrix}
a&b\\
b&c
\end{pmatrix}\zeta_1\zeta_1
+\begin{pmatrix}
2b&c+\frac{2b^2-ac}{a}\\
c+\frac{2b^2-ac}{a}&\frac{2bc}{a}
\end{pmatrix}\zeta_1\zeta_2
+\begin{pmatrix}
c&\frac{bc}{a}\\
\frac{bc}{a}&f
\end{pmatrix}\zeta_2\zeta_2\\
&=\begin{pmatrix}
a\zeta_1^2+2b\zeta_1\zeta_2+c\zeta_2^2 
&\frac{b}{a}(a\zeta_1^2+2b\zeta_1\zeta_2+c\zeta_2^2)\\
\frac{b}{a}(a\zeta_1^2+2b\zeta_1\zeta_2+c\zeta_2^2)
& \frac{c}{a}(a\zeta_1^2+2b\zeta_1\zeta_2+\frac{af}{c}\zeta_2^2)
\end{pmatrix}
\end{aligned}
\]
and
\[
\det(\ell(x,\zeta))=(a\zeta_1^2+2b\zeta_1\zeta_2+c\zeta_2^2)
(\frac{ac-b^2}{a^2})[a\zeta_1^2+2b\zeta_1\zeta_2+\frac{a^2f-b^2c}{ac-b^2}\zeta_2^2],
\]
respectively.
Since $ac-b^2>0$ and $f>\frac{2ab^2c-b^4}{a^3}$, this two polynomials
\[
a\zeta_1^2+2b\zeta_1\zeta_2+c\zeta_2^2;\quad a\zeta_1^2+2b\zeta_1\zeta_2+\frac{a^2f-b^2c}{ac-b^2}\zeta_2^2
\]
both have simple roots separably. Moreover, due to $f\neq\frac{c^2}{a}$, we have 
\[
\frac{a^2f-b^2c}{ac-b^2}\neq c.
\]
Therefore $\det(\ell(x,\zeta))=0$ in $\zeta$ has simple roots.
\end{proof}

\subsection*{Acknowledgement.} The authors sincerely thank Professor Jenn-Nan Wang for many helpful discussions and suggestions.
The second author would like to thank the support from National Center for Theoretical Sciences (NCTS) for his
stay in National Taiwan University, Taipei, Taiwan and the partial supports from Grant-in-Aid
for Scientific Research (15K21766 and 15H05740) of the Japan Society
for the Promotion of Science.


\bibliographystyle{plain}
\bibliography{ref}

\end{document}